\documentclass[11pt]{article}
\usepackage{amsmath,amsfonts,amssymb,amsthm}
\usepackage{enumerate}
\usepackage{xcolor}
\usepackage{url}
\usepackage{tcolorbox}
\usepackage{hyperref}
\usepackage{multicol, latexsym}
\usepackage{latexsym}
\usepackage{psfrag,import}
\usepackage{fullpage}
\usepackage{verbatim}
\usepackage[margin=1.1in]{geometry}

\usepackage{color}
\usepackage{epsfig}
\usepackage[outdir=./]{epstopdf}
\usepackage{hyperref}
\hypersetup{
    colorlinks=true,
    linkcolor=blue,
    filecolor=magenta,      
    urlcolor=cyan
}
\usepackage[title]{appendix}
\usepackage{geometry}
\usepackage{mathtools}
\usepackage{enumerate}
\usepackage{enumitem}   
\usepackage{multicol}
\usepackage{booktabs}
\usepackage{enumitem}
\usepackage{parcolumns}
\usepackage{thmtools}
\usepackage{xr}
\usepackage{epstopdf}
\usepackage{mathrsfs}
 \usepackage{subcaption}
 \usepackage{soul}
\usepackage{float}

\usepackage{comment}
\usepackage{authblk}
\usepackage{setspace}

\usepackage{cleveref}

\theoremstyle{definition}
\newtheorem{theorem}{Theorem}[section]

\newtheorem{proposition}[theorem]{Proposition}
\newtheorem{lemma}[theorem]{Lemma}

\theoremstyle{definition}
\newtheorem{definition}[theorem]{Definition}
\newtheorem{example}[theorem]{Example}

\newtheorem{remark}[theorem]{Remark}

\theoremstyle{remark}

\numberwithin{equation}{section}

\newcommand\RR{\mathbb{R}}

\newcommand\bla{\boldsymbol{\Lambda}}
\newcommand\by{\boldsymbol{y}}

\renewcommand\bf{\boldsymbol{f}}

\newcommand\bk{\boldsymbol{k}}
\newcommand\bw{\boldsymbol{w}}

\newcommand\bx{\boldsymbol{x}}
\newcommand\bv{\boldsymbol{v}}

\newcommand\bJ{\boldsymbol{J}}

\DeclareMathOperator{\codim} {codim}

\newcommand{\mK}{\mathcal{K}}
\newcommand{\dK}{\mathcal{K}_{\RR\text{-disg}}}

\newcommand{\pK}{\mathcal{K}_{\text{disg}}}
\newcommand{\mJ}{\mathcal{J}_{\RR}}

\newcommand{\eJ}{\mathcal{J}_{\textbf{0}}}
\newcommand{\mD}{\mathcal{D}_{\textbf{0}}}

\newcommand{\mS}{\mathcal{S}}

\usepackage{algorithm, algpseudocode}

\makeatletter
\newenvironment{breakablealgorithm}
  {
     \refstepcounter{algorithm}
     \hrule height.8pt depth0pt \kern2pt
     \renewcommand{\caption}[2][\relax]{
       {\raggedright\textbf{\fname@algorithm~\thealgorithm} ##2\par}%
       \ifx\relax##1\relax 
         \addcontentsline{loa}{algorithm}{\protect\numberline{\thealgorithm}##2}%
       \else 
         \addcontentsline{loa}{algorithm}{\protect\numberline{\thealgorithm}##1}%
       \fi
       \kern2pt\hrule\kern2pt
     }
  }{
     \kern2pt\hrule\relax
  }

\newcommand{\tJ}{\widetilde{\mathcal{J}}}
\newcommand{\tmJ}{\mathfrak{J}_{\RR}}

\newcommand\balpha{\boldsymbol{\alpha}}
\newcommand\bbeta{\boldsymbol{\beta}}

\newcommand{\defi}{\textbf}
\DeclareMathOperator{\spn}{span}
\DeclareMathOperator{\rank}{rank}

\makeatletter
\newcommand{\customlabel}[2]{%
\protected@write \@auxout {}{\string \newlabel {#1}{{#2}{}}}}
\makeatother

\begin{document}

\title{ 
The Computation of the Disguised Toric Locus of Reaction Networks
}

\author[1]{
         Gheorghe Craciun
}
\author[2]{
        Abhishek Deshpande
}
\author[3]{
        Jiaxin Jin
}
\affil[1]{\small Department of Mathematics and Department of Biomolecular Chemistry, \protect \\ University of Wisconsin-Madison}
\affil[2]{Center for Computational Natural Sciences and Bioinformatics, \protect \\
 International Institute of Information Technology Hyderabad}
\affil[3]{\small Department of Mathematics, University of Louisiana at Lafayette}

\date{} 

\maketitle

\begin{abstract}

\noindent
Mathematical models of reaction networks can exhibit very complex dynamics, including multistability, oscillations, and chaotic dynamics. On the other hand, under some additional assumptions on the network or on parameters values, these  models may actually be {\em toric dynamical systems}, which have remarkably stable  dynamics. The concept of ``\emph{disguised} toric dynamical system" was introduced in order to describe the phenomenon where a reaction network generates toric dynamics without actually being toric; such  systems enjoy all the stability properties of toric dynamical systems but with much fewer restrictions on the networks and parameter values.  
The \emph{disguised toric locus} is the set of parameter values for which the corresponding dynamical system is a disguised toric system.
Here we focus on providing a generic and efficient method for computing the dimension of the disguised toric locus of reaction networks.
Additionally, we illustrate our approach by applying it to some specific models of biological interaction networks, including Brusselator-type networks, Thomas-type networks, and  circadian clock networks.
\end{abstract}

\begin{NoHyper}
\tableofcontents
\end{NoHyper}

\section{Introduction}
\label{sec:intro}

Mathematical models of biochemical reaction networks are commonly described by polynomial dynamical systems \cite{yu2018mathematical,feinberg2019foundations, craciun2022homeostasis}.
Studying the dynamical properties of these networks is essential for understanding the behavior of chemical and biological systems \cite{yu2018mathematical,feinberg2019foundations, CraciunDickensteinShiuSturmfels2009, craciun2022autocatalytic}.
In general, analyzing these systems is a challenging problem. Classical nonlinear dynamical properties, such as multistability, oscillations, and chaotic dynamics, are difficult to investigate \cite{Ilyashenko2002, yu2018mathematical}.

There is a particular class of dynamical systems generated by reaction networks, known as \emph{complex-balanced systems} \cite{horn1972general} (or \emph{toric dynamical systems} \cite{CraciunDickensteinShiuSturmfels2009}), which are known for their remarkably robust dynamics.
In particular, all positive steady states in these systems are locally asymptotically stable~\cite{horn1972general, yu2018mathematical}. Further, these systems admit a unique positive steady state  within each affine invariant polyhedron; moreover, oscillations or chaotic dynamics are excluded for this class of systems \cite{horn1972general}.

However, the classical theory of complex-balanced systems has a limitation: to obtain a large set of parameter values (i.e., choices of reaction rate constants) that result in a complex-balanced system, the reaction network must satisfy additional graphical properties, namely \emph{weak reversibility} and \emph{low deficiency} (see \cite{yu2018mathematical} for details).
This limitation motivates the study of the notion of {\em dynamical equivalence}, which facilitates a significant relaxation of both restrictions.
Dynamical equivalence is based on the idea that two different reaction networks can generate the same dynamics for appropriately chosen parameter values. This phenomenon has also been referred to as \emph{macro-equivalence} \cite{horn1972general} or \emph{confoundability} \cite{craciun2008identifiability}.

The concept of a \emph{disguised toric locus} was first introduced in \cite{2022disguised}.
The disguised toric locus of a reaction network $G$ is the set of positive reaction rate vectors for which the corresponding dynamical system can be realized as a complex-balanced system by some network $G'$.
In other words, this locus consists of positive reaction rate vectors $\bk$ such that the mass-action system $(G, \bk)$ is dynamically equivalent to a complex-balanced system $(G', \bk')$.
Systems with reaction rate vectors in the disguised toric locus are called \emph{disguised toric systems}, as they are dynamically equivalent to complex-balanced systems.
Several general properties of the disguised toric locus have been established~\cite{disg_3, disg_1, disg_2}. For example, it was shown in~\cite{disg_1} that the disguised toric locus is path-connected.
In a recent paper~\cite{disg_3}, we derived a formula (see Theorem \ref{thm:dim_kisg_main}) for the dimension of the disguised toric locus.

In this paper, we develop a detailed approach for efficiently computing this formula. This is important and valuable because the formula established in~\cite{disg_3} is purely theoretical, and translating it into a computationally feasible method has remained unclear, especially when dealing with large reaction networks.
A key challenge in this context is computing the dimension of the term $\mJ(G', G)$ within the formula. This term represents the set of reaction rate constants $\bk'$ for the network $G'$ such that the system $(G', \bk')$ is both complex-balanced and $\RR$-realizable on the original network $G$. Specifically, we express it as
\begin{equation} \notag
\mJ(G',G) = \underbrace{ \mathcal{J} (G') }_{\text{complex-balanced}} \cap \ \underbrace{ \tmJ (G', G)}_{\text{$\RR$-realizable on $G$}}.
\end{equation}
Thus, computing the dimension of $\mJ(G', G)$ requires careful consideration of two distinct types of restrictions. One of the main results of this paper, Theorem~\ref{thm:main}, proves that these two restrictions are independent in terms of dimension. Moreover, we propose an algorithm (see Algorithm~\ref{algo:dim}) that efficiently computes the dimension of $\tmJ (G', G)$. This algorithm streamlines the process, making it feasible to apply to arbitrarily large networks.
In conjunction with the classic results on complex-balanced systems, Theorem~\ref{thm:main} provides a straightforward formula for computing the dimension of the disguised toric locus of an E-graph.

To summarize, our work advances the theoretical framework of the disguised toric locus, making it accessible for real-world applications in reaction networks across diverse fields. In the application section, we demonstrate the utility of our approach by applying it to several biological models, including Brusselator-type, Thomas-type, and circadian clock models.

\bigskip

\textbf{Structure of the paper.}
In Section~\ref{sec:background}, we review key definitions from reaction network theory, including dynamical equivalence, flux systems, and flux equivalence.
Section~\ref{sec:disguised_locus} revisits the definitions of the $\RR$-disguised toric locus and the disguised toric locus.
In particular, we recall the dimension formula for the two loci (see Theorem~\ref{thm:dim_kisg_main}). In Section~\ref{sec:main}, we focus on computing the dimension of the disguised toric locus, which decomposes into the sum of multiple terms. The main results of this paper, Theorem~\ref{thm:main} and Algorithm~\ref{algo:dim}, provide an effective method for computing each term.
In Section~\ref{sec:applications}, we illustrate our computation of the dimension of the disguised toric loci for the following biological models: Brusselator-type models, Thomas-type models, and Circadian clock models.
Finally, Section~\ref{sec:discussion} summarizes our findings and outlines directions for future research.

\bigskip

\textbf{Notation.}
Let $\mathbb{R}_{\geq 0}^n$ and $\mathbb{R}_{>0}^n$ denote the set of vectors in $\mathbb{R}^n$ with non-negative entries and positive entries, respectively. For vectors $\bx = (\bx_1, \ldots, \bx_n)^{\intercal}\in \RR^n_{>0}$ and $\by = (\by_1, \ldots, \by_n)^{\intercal} \in \RR^n$, define the following notation:
\begin{equation} \notag
\bx^{\by} = \bx_1^{y_{1}} \ldots \bx_n^{y_{n}}.
\end{equation}
For E-graphs (see Definition \ref{def:e-graph}), let $G$ denote an arbitrary E-graphs, and let $G'$ denote a weakly reversible E-graph.

\section{Background}
\label{sec:background}

This section overviews essential concepts and results in reaction network theory.

\subsection{Reaction Networks and Dynamical Equivalence}
\label{sec:reaction_networks}

We start by defining reaction networks and the concept of dynamical equivalence~\cite{craciun2008identifiability,horn1972general, kothari2024realizations,kothari2024endotactic,WR_DEF_THM,WR_df_1,deshpande2022source}, followed by a review of relevant terminology and basic properties.

\begin{definition}[\cite{craciun2015toric, craciun2019polynomial,craciun2020endotactic}]
\label{def:e-graph}

\begin{enumerate}[label=(\alph*)]
\item A \textbf{reaction network} $G = (V, E)$, also known as the \textbf{Euclidean embedded graph (or E-graph)}, is a directed graph in $\RR^n$, where $V \subset \mathbb{R}^n$ is a finite set of vertices and $E \subseteq V \times V$ is a finite set of edges.

\item A reaction in the network (denoted by $\by \rightarrow \by' \in E$) corresponds to a directed edge $(\by, \by') \in E$. Here, $\by$ is referred to as the \textbf{source vertex}, and $\by'$ as the \textbf{target vertex}. The difference vector $\by' - \by \in\mathbb{R}^n$ is called the \textbf{reaction vector}.
\end{enumerate}
\end{definition}

\begin{definition}

Let $G=(V, E)$ be an E-graph.

\begin{enumerate}[label=(\alph*)]
\item A set of vertices in $V$ is called a \textbf{linkage class} if it forms a connected component of $G$.
A linkage class is said to be \textbf{strongly connected} if every edge is part of a directed cycle. 
Furthermore, $G$ is said to be \textbf{weakly reversible} if every linkage class is strongly connected.

\item $G$ is called a \textbf{complete graph} if $\by \rightarrow \by' \in E$ for every pair of vertices $\by, \by' \in V$. For any E-graph $G$, there exists a complete graph (denoted by $G_c$) obtained by connecting every pair of vertices in $V$, and $G_c$ is referred to as the complete graph on $G$.

\item An E-graph $G' = (V', E')$ is said to be a \textbf{subgraph} of $G$ (denoted by $G' \subseteq G$) if $V' \subseteq V$ and $E' \subseteq E$.
Furthermore, if $G'$ is a weakly reversible subgraph of $G$, we denote this by $G' \sqsubseteq G$. 
\end{enumerate}
\end{definition}

\begin{definition}[\cite{adleman2014mathematics,guldberg1864studies,voit2015150,gunawardena2003chemical,yu2018mathematical,feinberg1979lectures}]

Let $G=(V, E)$ be an E-graph. Denote a \textbf{reaction rate vector} by
\[
\bk :=(k_{\by\to \by'})_{\by\to \by' \in E} \in \mathbb{R}_{>0}^{|E|}.
\]
Then $(G, \bk)$ generates a \textbf{mass-action dynamical system} on $\RR_{>0}^n$ given by
\begin{equation}
\label{def:mas_ds}
\frac{d\bx}{dt} = \displaystyle\sum_{\by \rightarrow \by' \in E}k_{\by\rightarrow\by'}{\bx}^{\by}(\by'-\by).
\end{equation}
The \textbf{stoichiometric subspace} of $G$ is defined as the span of its reaction vectors, that is,
\begin{equation} \notag
\mathcal{S}_G = \spn \{ \by' - \by: \by \rightarrow \by' \in E \}.
\end{equation}
Any solution to \eqref{def:mas_ds} with initial condition $\bx_0 \in \mathbb{R}_{>0}^n$ and $V \subset \mathbb{Z}_{\geq 0}^n$ is confined to the set $(\bx_0 + \mathcal{S}_G) \cap \mathbb{R}_{>0}^n$, and thus $(\bx_0 + \mathcal{S}_G) \cap \mathbb{R}_{>0}^n$ is called the \textbf{invariant polyhedron} of $\bx_0$.
\end{definition}

\begin{definition}

Let $(G, \bk)$ be a mass-action system \eqref{def:mas_ds}. 
A state $\bx^* \in \mathbb{R}^n_{>0}$ is called a \defi{positive steady state} of the system if 
\begin{equation} \notag
\displaystyle\sum_{\by\rightarrow \by' \in E } k_{\by\rightarrow\by'}{(\bx^*)}^{\by}(\by'-\by) = \mathbf{0}.
\end{equation}
A positive steady state $\bx^* \in \mathbb{R}^n_{>0}$ is called a \defi{complex-balanced steady state} of the system if for every vertex $\by_0 \in V$,
\begin{equation} \notag
\sum_{\by_0 \rightarrow \by \in E} k_{\by_0 \rightarrow \by} {(\bx^*)}^{\by_0}
= \sum_{\by' \rightarrow \by_0 \in E} k_{\by' \rightarrow \by_0} {(\bx^*)}^{\by'}.
\end{equation}
If a mass-action system $(G, \bk)$ admits a complex-balanced steady state, it is called a \textbf{toric dynamical system}.
\end{definition}

\begin{remark}

Toric dynamical systems are known for their robust graphical and dynamical properties. In \cite{horn1972general}, it was proven that all toric dynamical systems are weakly reversible, and all complex-balanced steady states are locally asymptotically stable and unique within each affine invariant polyhedron. Furthermore, toric dynamical systems are closely associated with the \emph{Global Attractor Conjecture}, which proposes that such systems have a globally attracting steady state within each stoichiometric compatibility class. Various special cases of this conjecture have been proved~\cite{anderson2011proof, pantea2012persistence, craciun2013persistence, boros2020permanence}. An approach for a proof of this conjecture in full generality  has been proposed in~\cite{craciun2015toric}, using the idea of toric differential inclusions~\cite{craciun2019quasi,craciun2020endotactic,ding2021minimal,ding2022minimal}.
\end{remark}

\begin{definition}
\label{def:de}

Let $(G, \bk)$ and $(G, \bk')$ be two mass-action systems.
Then $(G, \bk)$ and $(G', \bk')$ are said to be \defi{dynamically equivalent} if for every vertex\footnote{\label{footnote1} Note that when $\by_0 \not\in V$ or $\by_0 \not\in V'$, the corresponding side is considered as an empty sum.} $\by_0 \in V \cup V'$ we have
\begin{equation} \notag
\sum_{\by_0 \rightarrow \by\in E} k_{\by_0 \rightarrow \by} (\by - \by_0) 
= \sum_{\by_0 \rightarrow \by'\in E'} k'_{\by_0 \rightarrow\by'} (\by' - \by_0).
\end{equation}
We denote $(G,\bk) \sim (G', \bk')$ if the mass-action systems $(G,\bk)$ and $(G',\bk')$ are dynamically equivalent.
\end{definition}

\begin{definition} 
\label{def:d0}

Let $G = (V, E)$ be an E-graph let $\bla = (\Lambda_{\by \to \by'})_{\by \to \by' \in E} \in \RR^{|E|}$. The set $\mD (G)$ is defined as
\begin{equation} \notag
\mD (G):=
\{\bla \in \RR^{|E|} \ \Big| \ \sum_{\by_0 \to \by \in E} \Lambda_{\by_0  \to \by} (\by - \by_0) = \mathbf{0} \ \text{ for every vertex } \by_0 \in V
\}.
\end{equation}
\end{definition} 

\begin{remark}
\label{rmk:d0}

Given an E-graph $G = (V, E)$, Definition \ref{def:d0} implies that the set $\mD (G)$ is a linear subspace. For each vertex $\by_0 \in V$ and the corresponding reactions $\{ \by_0 \to \by \}_{\by_0 \to \by \in E}$, consider the matrix $M_{\by_0}$ whose columns are the reaction vectors associated with these reactions. Then, for every $\bla \in \mD (G)$, $\bla_{\by_0 \to \by}$ belongs to the kernel of the matrix $M_{\by_0}$.
\end{remark}

\begin{lemma}[\cite{disg_2}]
\label{lem:d0}

Let $(G, \bk)$ and $(G, \bk')$ be two mass-action systems. Then $(G, \bk) \sim (G, \bk')$ if and only if $\bk' - \bk \in \mD (G)$.    
\end{lemma}

Lemma \ref{lem:d0} indicates that $\mD(G) \subseteq \RR^{|E|}$ consists of the rate vectors on the E-graph $G$ that preserve the dynamical system under the dynamical equivalence.
Specifically, for any mass-action system $(G, \bk)$ and any vector $\bla \in \mD(G)$, if $\bk + \bla \in \RR^{|E|}_{>0}$, then the systems $(G, \bk)$ and $(G, \bk + \bla)$ are dynamically equivalent.
The following example illustrates the computation of $\mD$ for a given E-graph.

\begin{example}
\label{ex:d0}

Figure~\ref{fig:running_example} illustrates two E-graphs $G = (V, E)$ and $G' = (V', E')$. We now present the computation of $\mD$ for both E-graphs. 

\begin{figure}[!ht]
\centering
\includegraphics[scale=0.4]{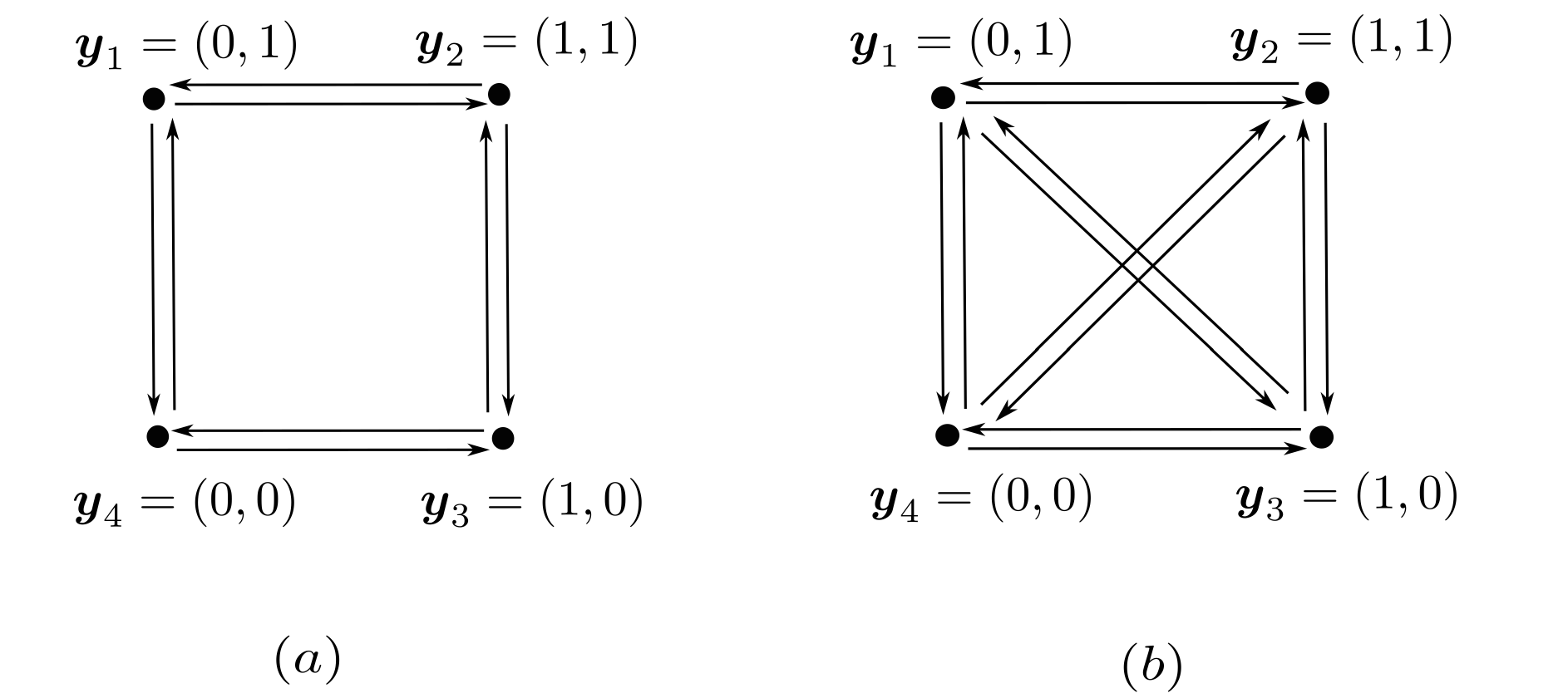}
\caption{Two E-graphs $G = (V, E)$ and $G' = (V', E')$.}
\label{fig:running_example}
\end{figure} 

(a) 
For each vertex $\by_i \in V$, the corresponding reaction vectors $\{ \by_j - \by_i \}_{\by_i \to \by_j \in E}$ are linearly independent. Remark~\ref{rmk:d0} implies that the kernel of the matrix $M_{\by_i}$ contains only the zero vector. Therefore,
\[
\mD (G) = \{ \mathbf{0} \}.
\]

(b) 
Following Remark~\ref{rmk:d0}, the matrix $M_{\by_i}$ can be constructed for each vertex $\by_i$. The kernels of these matrices are given as follows:

\begin{enumerate}
\item[(i)] For the vertex $\by_1$, the non-zero kernel vector of $M_{\by_1}$ is given by:
\begin{equation} \notag
\bv_{1, r} = 
\begin{cases}
1, & \text{ if } r = \by_1 \to \by_2 \text{ or } \by_1 \to \by_4, \\
-1, & \text{ if } r = \by_1 \to \by_3.
\end{cases}
\end{equation}

\item[(ii)]  For the vertex $\by_2$, the non-zero kernel vector of $M_{\by_2}$ is given by:
\begin{equation} \notag
\bv_{2, r} = 
\begin{cases}
1, & \text{ if } r = \by_2 \to \by_1 \text{ or } \by_2 \to \by_3, \\
-1, & \text{ if } r = \by_2 \to \by_4.
\end{cases}
\end{equation}

\item[(iii)] For the vertex $\by_3$, the non-zero kernel vector of $M_{\by_3}$ is given by:
\begin{equation} \notag
\bv_{3, r} = 
\begin{cases}
1, & \text{ if } r = \by_3 \to \by_2 \text{ or } \by_3 \to \by_4, \\
-1, & \text{ if } r = \by_3 \to \by_1.
\end{cases} 
\end{equation}

\item[(iv)] For the vertex $\by_4$, the non-zero kernel vector of $M_{\by_4}$ is given by:
\begin{equation} \notag
\bv_{4, r} = 
\begin{cases}
1, & \text{ if } r = \by_4 \to \by_1 \text{ or } \by_4 \to \by_3, \\
-1, & \text{ if } r = \by_4 \to \by_2.
\end{cases}
\end{equation}
\end{enumerate}

Remark~\ref{rmk:d0} shows that $\bla_{\by_i \to \by}$ belongs to the kernel of the matrix $M_{\by_i}$ for every $\bla \in \mD (G)$.
Let $\bv'_1, \ldots. \bv'_4$ denote the vectors obtained by extending the dimension of $\bv_1, \ldots, \bv_4$ to $|E'|$ by appending zeros, respectively.
Therefore,
\[
\mD(G') = \spn \{ \bv'_1, \bv'_2, \bv'_3, \bv'_4 \}
\ \text{ and } \
\dim(\mD (G')) = 4.
\]
\qed
\end{example}

\subsection{Flux Systems and Flux Equivalence} 
\label{sec:flux_systems}

In this subsection, we introduce flux systems and the concept of flux equivalence.

\begin{definition}
 \label{def:f_s}
 
Let $G=(V, E)$ be an E-graph. Denote a \textbf{flux vector} by
\[
\bJ :=(J_{\by\to \by'})_{\by\to \by' \in E} \in \mathbb{R}_{>0}^{|E|}.
\]
Then $(G, \bJ)$ generates a \textbf{flux system} on $\RR_{>0}^n$ given by
\begin{equation}  \label{eq:f_s}
\frac{d\bx}{dt} = \displaystyle\sum_{\by \rightarrow \by' \in E} J_{\by \rightarrow \by'} (\by'-\by).
\end{equation}
\end{definition}

\begin{definition}

Let $G = (V, E)$ be an E-graph. 
A flux vector $\bJ^* \in \mathbb{R}^{|E|}_{>0}$ is called a \defi{steady flux vector} of $G$ if 
\begin{equation} \notag
\displaystyle\sum_{\by\rightarrow \by' \in E } J^*_{\by\rightarrow\by'} (\by'-\by) = \mathbf{0}.
\end{equation}
A steady flux vector $\bJ^* \in \mathbb{R}^{|E|}_{>0}$ is called a \defi{complex-balanced flux vector} of $G$ if for every vertex $\by_0 \in V$,
\begin{equation} \notag
\sum_{\by_0 \rightarrow \by \in E} J^*_{\by_0 \rightarrow \by} = \sum_{\by' \rightarrow \by_0 \in E} J^*_{\by' \rightarrow \by_0}.
\end{equation}
Let $\mathcal{J}(G)$ denote the set of all complex-balanced flux vectors of $G$, defined as follows:
\begin{equation} \notag
\mathcal{J}(G):=
\{\bJ \in \RR_{>0}^{|E} \ \big| \ \bJ \text{ is a complex-balanced flux vector of $G$} \}.
\end{equation}
\end{definition}

\begin{definition}

Let $(G,\bJ)$ and $(G', \bJ')$ be two flux systems. Then $(G,\bJ)$ and $(G', \bJ')$ are said to be \defi{flux equivalent} if for every vertex\footref{footnote1} $\by_0 \in V \cup V'$,
\begin{equation} \notag
\sum_{\by_0 \to \by \in E} J_{\by_0 \to \by} (\by - \by_0) 
= \sum_{\by_0 \to \by' \in E'} J'_{\by_0 \to \by'} (\by' - \by_0).
\end{equation}
We denote $(G, \bJ) \sim (G', \bJ')$ if two flux systems $(G, \bJ)$ and $(G', \bJ')$ are flux equivalent.
\end{definition}

\begin{definition} 
\label{def:j0}

Let $G = (V, E)$ be an E-graph and let $\bJ = ({J}_{\by \to \by'})_{\by \to \by' \in E} \in \RR^{|E|}$. Recall $\mD (G)$ in Definition \ref{def:d0}, the set $\eJ (G)$ is defined as
\begin{equation} \notag
\eJ (G) :=
\{{\bJ} \in \mD (G) \ \Big| \ \sum_{\by \to \by_0 \in E} {J}_{\by \to \by_0} 
= \sum_{\by_0 \to \by' \in E} {J}_{\by_0 \to \by'} \ \text{for every vertex } \by_0 \in V
\}.
\end{equation}
\end{definition} 

\begin{remark}
\label{rmk:j0}

Given an E-graph $G = (V, E)$, Remark \ref{rmk:d0} and Definition \ref{def:j0} show that the set $\eJ (G)$ is a subset of $\mD (G)$ and a linear subspace of $\mathbb{R}^{|E|}$.
Hence, $\eJ (G)$ satisfies the properties of $\mD (G)$ introduced in Remark \ref{rmk:d0}.
Suppose $V = \{ \by_1, \ldots, \by_m \}$ and index the $|E|$ reactions in $G$. Consider the matrix $B \in \mathbb{R}^{m \times |E|}$ as follows:
\[
B_{ij} =
\begin{cases}
1 & \text{ if $\by_i$ is the source vertex in the $j$-th reaction}, \\
-1 & \text{ if $\by_i$ is the target vertex in the $j$-th reaction}, \\
0 & \text{ otherwise}.
\end{cases}
\]
Then $\eJ (G)$ belongs to the kernel of the matrix $B$.
\end{remark}

\begin{lemma}[\cite{disg_2}]
\label{lem:j0}
Let $(G, \bJ)$ and $(G, \bJ')$ be two flux systems. Then 
\begin{enumerate}
\item[(a)] $(G, \bJ) \sim (G, \bJ')$ if and only if $\bJ' - \bJ \in \mD (G)$.

\item[(b)] If both $\bJ$ and $\bJ'$ are complex-balanced flux vector of $G$, then $(G, \bJ) \sim (G, \bJ')$ if and only if $\bJ' - \bJ \in \eJ(G)$.
\end{enumerate} 
\end{lemma}

Similarly to Lemma \ref{lem:d0}, Lemma \ref{lem:j0} shows that $\eJ (G) \subseteq \RR^{|E|}$ consists of the flux vectors on the E-graph $G$ that preserve both the complex-balanced property and the dynamical system under dynamical equivalence.
The following example illustrates the computation of $\eJ$ for a given E-graph.

\begin{example}
\label{ex:j0}

Revisit two E-graphs $G = (V, E)$ and $G' = (V', E')$ in Figure \ref{fig:running_example}.
The computation of $\eJ$ for both E-graphs is presented below.

\smallskip

(a) 
Definition \ref{def:j0} implies that $\eJ (G) \subseteq \mD (G)$. From Example \ref{ex:d0}, $\mD (G) = \{ \mathbf{0} \}$ and thus
\[
\eJ (G) = \{ \mathbf{0} \}.
\]

(b)
Recall from Example \ref{ex:d0} that $\mD(G') = \spn \{ \bv'_1, \bv'_2, \bv'_3, \bv'_4 \}$. Based on Remark~\ref{rmk:j0}, we construct the matrix $B_{G'}$ associated with $G'$, and derive that
\[
\{ \bv'_1 + \bv'_2, \ \ \bv'_1 - \bv'_3, \ \ \bv'_1 + \bv'_4 \} \subseteq \ker (B_{G'}) \cap \mD(G').
\]
Therefore,
\[
\eJ (G') = \spn \{ \bv_1 + \bv_2,  \bv_1 - \bv_3, \ \bv_1 + \bv_4 \}
\ \text{ and } \
\dim (\eJ(G')) =  3.
\]
\qed
\end{example}

At the end of this section, we present the following proposition, which establishes the relationship between dynamical equivalence and flux equivalence.

\begin{proposition}[\cite{craciun2020efficient}]
\label{prop:craciun2020efficient}

Let $(G, \bk)$ and $(G, \bk')$ be two mass-action systems. For $\bx \in \RR_{>0}^n$, define the flux vector $\bJ (\bx) = (J_{\by \to \by'})_{\by \to \by' \in E}$ on $G$, such that for every $\by \to \by' \in E$,
\begin{equation} \notag
J_{\by \to \by'} = k_{\by \to \by'} \bx^{\by}.
\end{equation}
Further, define the flux vector $\bJ' (\bx) = (J'_{\by \to \by'})_{\by \to \by' \in E'}$ on $G'$, such that for every $\by \to \by' \in E$,
\begin{equation} \notag
J'_{\by \to \by'} = k'_{\by \to \by'} \bx^{\by}.
\end{equation} 
Then the following are equivalent:
\begin{enumerate}
\item[(a)] The mass-action systems $(G, \bk)$ and $(G', \bk')$ are dynamically equivalent.

\item[(b)] The flux systems $(G, \bJ(\bx))$ and $(G', \bJ')$ are flux equivalent for all $\bx \in \RR_{>0}^n$.

\item[(c)] The flux systems $(G, \bJ(\bx))$ and $(G', \bJ'(\bx))$ are flux equivalent for some $\bx \in \RR_{>0}^n$
\end{enumerate}
\end{proposition}

\section{Disguised Toric Locus and \texorpdfstring{$\RR$}{R}-disguised Toric Locus} 
\label{sec:disguised_locus}

In this section, we introduce the key concepts of this paper: the disguised toric locus $\pK(G)$ and the $\RR$-disguised toric locus $\dK(G)$.
Additionally, we outline the method for determining the dimensions of $\dK(G)$ and $\pK(G)$.


\begin{definition}[\cite{disg_2}]
\label{def:mas_realizable}
Let $G=(V, E)$ be an E-graph. A dynamical system 
\[
\frac{\mathrm{d} \bx}{\mathrm{d} t} 
= \bf (\bx),
\]
is said to be \defi{$\RR$-realizable}
on $G$, if there exists some $\bk \in \mathbb{R}^{|E|}$ such that
\begin{equation} \label{eq:realization}
\bf (\bx) =
\sum_{\by \rightarrow \by' \in E}k_{\by_i \rightarrow \by_j} \bx^{\by_i}(\by_j - \by_i).
\end{equation}
In addition, the dynamical system is said to be \defi{realizable} 
on $G$ if $\bk \in \mathbb{R}^{|E|}_{>0}$ in \eqref{eq:realization}.
\end{definition}

\begin{definition}[\cite{CraciunDickensteinShiuSturmfels2009}]

Let $G=(V, E)$ be an E-graph. The \defi{toric locus} of $G$ is defined as
\begin{equation} \notag
\mK (G) := \{ \bk \in \mathbb{R}_{>0}^{|E|} \ \big| \ (G, \bk) \ \text{is a toric dynamical system} \}.
\end{equation}
A dynamical system is said to be \defi{disguised toric} (or has a \defi{toric realization}) on $G$ if it is realizable on $G$ for some $\bk \in \mK (G)$. 
\end{definition}

We are now prepared to define the disguised toric locus and the $\RR$-disguised toric locus.

\begin{definition}
\label{def:de_realizable}

Let $G = (V, E)$ and $G' = (V', E')$ be two E-graphs.
The set $\dK(G, G')$ is defined as 
\begin{equation} \notag
\dK(G, G') := \{ \bk \in \mathbb{R}^{|E|} \ \big| \ \text{the dynamical system } (G, \bk)\footnote{\label{footnote2} Note that when $\bk \not\in \mathbb{R}^{|E|}_{>0}$, the dynamical system $(G, \bk)$ follows \eqref{def:mas_ds}.} \ \text{is disguised toric on } G' \}.
\end{equation} 
The \defi{$\RR$-disguised toric locus} of $G$ is given by
\begin{equation} \notag
\dK(G) := \displaystyle\bigcup_{G' \sqsubseteq G_{c}} \ \dK(G, G').
\end{equation}
Define $\pK (G, G') := \dK(G, G') \cap \mathbb{R}^{|E|}_{>0}$, the \defi{disguised toric locus} of $G$ is given by
\begin{equation} \notag
\pK (G) := \displaystyle\bigcup_{G' \sqsubseteq G_{c}} \ \pK(G, G').
\end{equation}
\end{definition}

The set $\dK(G, G')$ consists of all rate vectors $\bk$ for which the system $(G, \bk)$ can be realized in $G'$ with reaction rate constants belonging to the toric locus of $G'$.
The $\RR$-disguised toric locus $\dK(G)$ is defined as the union of all sets $\dK(G, G')$, where $G'$ is a weakly reversible subgraph of $G_{c}$.
Similarly, the set $\pK(G, G')$ and the \textbf{disguised toric locus} $\pK(G)$ are defined analogously to $\dK(G, G')$ and $\dK(G)$, but with the additional requirement that the rate vectors satisfy $\bk \in \mathbb{R}^{|E|}_{>0}$.

\begin{remark} 

The restriction to graphs $G'$ satisfying $G' \sqsubseteq G_{c}$ when defining the disguised toric locus and the $\RR$-disguised toric locus is well-justified.
In principle, $\dK (G)$ is intended to include all vectors $\bk \in \mathbb{R}^{|E|}$ such that the system $(G,\bk)$ has a toric realization with respect to any (arbitrary) graph $\tilde{G}$.
However, as shown in~\cite{craciun2020efficient}, if a dynamical system generated by $G$ is disguised toric on some graph $\tilde{G}$, then there exists a graph $G' \sqsubseteq G_{c}$ that can also give rise to a toric realization of the same dynamical system. 
Consequently, restricting $G'$ to weakly reversible subgraphs of $G_c$ does not alter the definition of $\dK(G)$.
\end{remark}

To compute the dimension of the disguised toric locus and the $\RR$-disguised toric locus, we start by defining the following set.

\begin{definition}
\label{def:flux_realizable}

Let $(G', \bJ')$ be a flux system. It is said to be \defi{$\RR$-realizable} on $G$ if there exists some $\bJ \in \mathbb{R}^{|E|}$, such that for every vertex\footref{footnote1} $\by_0 \in V \cup V'$,
\begin{equation} \notag
\sum_{\by_0 \to \by \in E} J_{\by_0 \to \by} 
(\by - \by_0) 
= \sum_{\by_0 \to \by' \in E'} J'_{\by_0 \to \by'} 
(\by' - \by_0).
\end{equation}
The set $\mJ (G', G)$ is defined as
\begin{equation} \notag
\mJ (G', G) := \{ \bJ' \in \mathcal{J} (G') \ \big| \ \text{the flux system } (G', \bJ') \text{ is $\RR$-realizable on } G \}.
\end{equation}
\end{definition} 

\begin{example}

Recall the E-graphs $G = (V, E)$ and $G' = (V', E')$ in Figure \ref{fig:running_example}. For the edges in $G'$ that also appear in $G$, their corresponding flux vectors are certainly $\RR$-realizable on $G$.  
Moreover, $G'$ contains four additional diagonal edges $\{ \by_1 \rightleftharpoons \by_3, \ \by_2 \rightleftharpoons \by_4 \}$ that are not in $G$.
By direct computation, each of these flux vectors is $\RR$-realizable on $G$. For example, 
\[
J_{\by_1 \to \by_3} (\by_3 - \by_1) = J_{\by_1 \to \by_3} (\by_2 - \by_1) + J_{\by_1 \to \by_3} (\by_4 - \by_1).
\]
Therefore, we conclude that $\mJ (G', G) = \mathcal{J} (G')$.
\qed
\end{example}

In \cite{disg_3}, we explored the connection between $\RR$-realizable rate vectors $\dK(G, G')$ and $\RR$-realizable flux vectors $\mJ(G', G)$.
Specifically, we  proved the existence of a homomorphism $\varphi$ that maps between two product spaces as follows:
\[
\varphi: \dK(G, G') \times \mD(G) \to \mJ(G', G) \times \mS_{G'} \times \eJ(G').
\]
This result enables us to derive the following theorem.

\begin{theorem}[\cite{disg_3}]
\label{thm:dim_kisg}

Let $G = (V, E)$ be an E-graph and let $G' = (V', E')$ be a weakly reversible E-graph with its stoichiometric subspace $\mS_{G'}$.

\begin{enumerate}[label=(\alph*)]
\item\label{part_a} Consider $\dK (G, G')$ from Definition~\ref{def:de_realizable}, then
\begin{equation} \notag
\begin{split} 
& \dim(\dK(G, G')) 
= \dim (\mJ(G', G)) + \dim (\mS_{G'})  + \dim(\eJ(G')) - \dim(\mD(G)),
\end{split}
\end{equation}
where $\mJ (G',G)$, $\mD(G)$, and $\eJ(G')$ are defined in Definitions~\ref{def:flux_realizable}, \ref{def:d0}, and \ref{def:j0}, respectively.

\item\label{part_b} Consider $\pK (G, G')$ from Definition~\ref{def:de_realizable} and assume that $\pK (G, G') \neq \emptyset$. Then
\begin{equation} \notag
\dim(\pK (G, G')) = \dim(\dK(G, G')).
\end{equation}
\end{enumerate}
\end{theorem}

\begin{remark}[\cite{disg_2}]
\label{rmk:semi_algebaic}

Given two E-graphs $G = (V, E)$ and $G' = (V', E')$, 
both $\dK(G, G')$ and $\pK(G, G')$ are semialgebraic sets.
On a dense open subset of $\dK(G, G')$ or $\pK(G, G')$, these sets are locally submanifolds. The dimension of $\dK(G, G')$ or $\pK(G, G')$ is defined as the largest dimension at points where the sets are submanifolds.
\end{remark}

As a consequence of Theorem \ref{thm:dim_kisg} and Definition \ref{def:de_realizable}, the dimensions of $\dK(G)$ and $\pK(G)$ are given as follows. 

\begin{theorem}[\cite{disg_3}]
\label{thm:dim_kisg_main}

Let $G = (V, E)$ be an E-graph.

\begin{enumerate}[label=(\alph*)]
\item Consider $\dK(G)$ from Definition~\ref{def:de_realizable}. Then
\begin{equation} \notag
\dim (\dK(G) )
= \max_{G'\sqsubseteq G_c} 
\Big\{ \dim (\mJ(G',G)) + \dim (\mS_{G'})  + \dim(\eJ(G')) - \dim(\mD(G)) 
\Big\},
\end{equation}
where $\mJ (G',G)$, $\mD(G)$, and $\eJ(G')$ are defined in Definitions \ref{def:flux_realizable}, \ref{def:d0}, and \ref{def:j0}, respectively.

\item Further, consider $\pK (G)$ from Definition~\ref{def:de_realizable}. Assume that $\pK (G) \neq \emptyset$, then
\begin{equation} \notag
\begin{split}
& \dim (\pK(G) )
\\& = \max_{ \substack{ G'\sqsubseteq G_c, \\  \pK(G, G') \neq \emptyset } } 
\Big\{ \dim (\mJ(G',G)) + \dim (\mS_{G'})  + \dim(\eJ(G')) - \dim(\mD(G)) 
\Big\}.
\end{split}
\end{equation}
\end{enumerate}
\end{theorem}

\section{Computing the Dimension of the Disguised Toric Locus}
\label{sec:main}

Theorem \ref{thm:dim_kisg_main} provides a formula for $\dim (\pK(G))$. 
Definition \ref{def:mas_ds}, along with Remarks \ref{rmk:d0} and \ref{rmk:j0}, offers a framework for computing $\dim (\mS_{G'})$, $\dim(\mD(G))$, and $\dim(\eJ(G'))$, respectively.
However, the computation of $\dim (\mJ(G', G))$ in this formula remains unresolved, preventing the full determination of the disguised toric locus dimension.

In this section, we address this gap by providing a method for computing $\dim (\mJ(G', G))$, which is the main result of this paper. We begin by defining two sets as follows.

\begin{definition}
\label{def:tilde_j_G}

Let $G = (V, E)$ and $G' = (V', E')$ be two E-graphs.

\begin{enumerate}[label=(\alph*)]
\item The set $\tJ (G)$ is defined as
\begin{equation} \notag
\tJ (G) := \{ \bJ \in \RR^{|E|} \ \big| \ \sum_{\by \to \by_0 \in E} {J}_{\by \to \by_0} 
= \sum_{\by_0 \to \by' \in E} {J}_{\by_0 \to \by'} \ \text{for every vertex } \by_0 \in V \}.
\end{equation}

\item  The set $\tmJ (G', G)$ is defined as
\[
\tmJ (G', G) := \{ \bJ' \in \RR^{|E'|} \ \big| \ \text{the flux system } (G', \bJ')\footnote{\label{footnote3} Note that when $\bJ \not\in \RR^{|E'|}_{>0}$, the flux system $(G', \bJ')$ follows \eqref{eq:f_s}.} \ \text{is $\RR$-realizable on } G \}.
\]
\end{enumerate}
\end{definition}

\begin{lemma}
\label{lem:jr_g'_g}

Let $G = (V, E)$ and $G' = (V', E')$ be two E-graphs.

\begin{enumerate}[label=(\alph*)]
\item Two sets $\tJ (G')$ and $\tmJ (G', G)$ are both linear subspaces of $\RR^{|E'|}$.

\item Consider $\mJ (G', G)$ from Definition \ref{def:flux_realizable}. Assume that $\mJ (G', G) \neq \emptyset$, then
\begin{equation} \label{eq:jr_g'_g}
\dim (\mJ (G', G)) = \dim \big( \tJ (G') \cap \tmJ (G', G) \big).
\end{equation}
\end{enumerate}
\end{lemma}

\begin{proof}

(a)
By direct computation, Definition \ref{def:tilde_j_G} shows that the set $\tJ (G')$ forms a linear subspace of $\RR^{|E'|}$.
We now prove that $\tmJ (G', G)$ is a linear subspace of $\RR^{|E'|}$.

Suppose $\bJ'_1, \bJ'_2 \in \tmJ (G', G)$ and two real values $\alpha_1, \alpha_2 \in \mathbb{R}$.
Since both flux systems $(G', \bJ'_1)$ and $(G', \bJ'_2)$ are $\RR$-realizable on G, Definition \ref{def:flux_realizable} shows that there exist $\bJ_1, \bJ_2 \in \mathbb{R}^{|E|}$, such that for every vertex $\by_0 \in V \cup V'$,
\begin{equation} \label{eq1:jr_g'_g}
\begin{split}
\sum_{\by_0 \to \by \in E} J_{1, \by_0 \to \by} 
(\by - \by_0) 
& = \sum_{\by_0 \to \by' \in E'} J'_{1, \by_0 \to \by'} 
(\by' - \by_0), 
\\ \sum_{\by_0 \to \by \in E} J_{2, \by_0 \to \by} 
(\by - \by_0) 
& = \sum_{\by_0 \to \by' \in E'} J'_{2, \by_0 \to \by'} 
(\by' - \by_0).
\end{split}
\end{equation}
Consider $\bJ' = \alpha_1 \bJ'_1 + \alpha_2 \bJ'_2 \in \RR^{|E'|}$. From \eqref{eq1:jr_g'_g}, we obtain that
\begin{equation} \notag
\sum_{\by_0 \to \by \in E} (\alpha_1 J_{1, \by_0 \to \by} + \alpha_2 J_{2, \by_0 \to \by}) (\by - \by_0) 
= \sum_{\by_0 \to \by' \in E'} J'_{1, \by_0 \to \by'} 
(\by' - \by_0).
\end{equation}
Thus, $\bJ' \in \tmJ (G', G)$ and $\tmJ (G', G)$ is a linear subspace of $\RR^{|E'|}$.

\smallskip

(b) 
From Definition \ref{def:flux_realizable}, we have
\begin{equation} \label{eq2:jr_g'_g}
\mJ (G', G) = \big( \tJ (G') \cap \tmJ (G', G) \big) \cap \RR^{|E'|}_{>0}.
\end{equation}
Part (a) shows both $\tJ (G')$ and $\tmJ (G', G)$ are linear subspaces of $\RR^{|E'|}$. This implies that $\tJ (G') \cap \tmJ (G', G)$ is also a linear subspace of $\RR^{|E'|}$.
From \eqref{eq2:jr_g'_g} and the assumption that $\mJ (G', G) \neq \emptyset$, $\mJ (G', G)$ is an open cone and thus we conclude \eqref{eq:jr_g'_g}.
\end{proof}

\smallskip

\textbf{Notation:} 
We introduce the following notation, which will be used in the rest of this section and in Section \ref{sec:applications}.

(a) Assume the E-graph $G' = (V', E')$ has $|V'| = m$ vertices, $|E'| = r$ reactions, and consists of $\ell \geq 1$ linkage classes, denoted by $L_1 = (V_1, E_1), \ldots, L_{\ell} = (V_{\ell}, E_{\ell})$, such that
\[
V' = V_1 \sqcup V_2 \sqcup \cdots \sqcup V_{\ell}
\ \text{ and } \
E' = E_1 \sqcup E_2 \sqcup \cdots \sqcup E_{\ell},
\]
where $V_i = \{ \by_{i,1}, \ldots, \by_{i, m_i} \}$ for each $1 \leq i \leq \ell$.
Further, assume that $|E_i| = r_i$ for each $1 \leq i \leq \ell$, and that the reactions in each linkage class are indexed accordingly.

\smallskip

(b)
From Lemma \ref{lem:jr_g'_g}, $\tJ (G')$ and $\tmJ (G', G)$ are both linear subspaces of $\RR^{r}$.
Let their codimensions be denoted by
\[
\codim \big( \tJ (G') \big) = d_1,
\ \
\codim \big( \tmJ (G', G) \big) = d_2,
\]
and assume $\tJ (G')^{\perp}$ has a basis $\{ \balpha_1, \ldots, \balpha_{d_1} \}$, and $\tmJ (G', G)^{\perp}$ has a basis $\{ \bbeta_1, \ldots, \bbeta_{d_2} \}$. 
Consider two matrices $A, B$ defined as follows:
\begin{equation} \label{def1:codimen_sum}
A = 
\begin{pmatrix}
\balpha^{\intercal}_1 \\
\vdots \\
\balpha^{\intercal}_{d_1}
\end{pmatrix} \in \mathbb{R}^{d_1 \times r}
\ \text{ and } \
B = 
\begin{pmatrix}
\bbeta^{\intercal}_1 \\
\vdots \\
\bbeta^{\intercal}_{d_2}
\end{pmatrix} \in \mathbb{R}^{d_2 \times r}.
\end{equation}
Thus, two linear subspaces $\tJ (G')$ and $\tmJ (G', G)$ can be expressed as
\begin{equation} \label{def2:codimen_sum}
\begin{split}
\tJ (G') & = \{ \bJ' \in \RR^{r} \ \big| \ A \bJ' = \mathbf{0} \},
\\ \tmJ (G', G) & = \{ \bJ' \in \RR^{r} \ \big| \ B \bJ' = \mathbf{0} \}.
\end{split}
\end{equation}
\qed

\smallskip

We begin with an intermediate lemma, which will be used in the proof of the main Theorem~\ref{thm:main}.

\begin{lemma}
\label{lem:intermediate}

Let $G' = (V', E')$ be a weakly reversible E-graph and let $G = (V, E)$ be an E-graph.
Given a vector 
\[
\bJ' = (\bJ'_{\by_i \to \by_j})_{\by_i \to \by_j \in E'} \in \tJ (G') \cap \tmJ (G', G),
\]
let $\bw = (w_1, \ldots, w_m) \in \mathbb{R}^m$ be an arbitrary real vector. Consider the following expression:
\begin{equation} \label{eq2:codimen_sum}
\bJ'_{\bw} = (\bJ'_{\bw, \by_i \to \by_j})_{\by_i \to \by_j \in E'}
\ \text{ with } \
\bJ'_{\bw, \by_i \to \by_j} = \bJ'_{\by_i \to \by_j} w_i.
\end{equation}
Then, we have the following property:
\begin{enumerate}[label=(\alph*)]
\item $\bJ'_{\bw} \in \tmJ (G', G)$ for any $\bw = (w_1, \ldots, w_m) \in \mathbb{R}^m$.

\item $\bJ'_{\bw} \in \tJ (G')$ if and only if
$\bw \in \ker (\tilde{J})$, where $\tilde{J} \in \mathbb{R}^{m \times m}$ is the Kirchoff matrix defined as
\begin{equation} \label{def3:codimen_sum}
\tilde{J}_{ij} =
\begin{cases}
- \sum\limits_{\by_i \to \by \in E'} J'_{\by_i \to \by} & \text{ if $i = j$}, \\[5pt]
\quad J'_{\by_j \to \by_i} & \text{ if $i \neq j$ and $\by_j \to \by_i \in E'$}, \\[5pt]
\quad 0 & \text{ otherwise}.
\end{cases}
\end{equation}
\end{enumerate}
\end{lemma}

\begin{proof}

(a)
Since $\bJ' \in \tmJ (G', G)$, from Definition \ref{def:flux_realizable} there exists some $\bJ \in \mathbb{R}^{|E|}$, such that for every vertex $\by_i \in V \cup V'$,
\begin{equation} \notag
\sum_{\by_i \to \by \in E} J_{\by_i \to \by} 
(\by - \by_i) 
= \sum_{\by_i \to \by' \in E'} J'_{\by_i \to \by'} 
(\by' - \by_i).
\end{equation}
This, together with \eqref{eq2:codimen_sum}, implies that for every vertex $\by_i \in V \cup V'$,
\begin{equation} \notag
\sum_{\by_i \to \by \in E} J_{\by_i \to \by} w_i
(\by - \by_i) 
= \sum_{\by_i \to \by' \in E'} J'_{\by_i \to \by'} w_i
(\by' - \by_i)
= \sum_{\by_0 \to \by' \in E'} J'_{\bw, \by_i \to \by'} 
(\by' - \by_i).
\end{equation}
Thus, we prove part (a).
Recall the notation in \eqref{def1:codimen_sum}; this further indicates that
\begin{equation} \label{eq9:codimen_sum}
\bJ'_{\bw} \in \ker \left( B = 
\begin{pmatrix}
\bbeta^{\intercal}_1 \\
\vdots \\
\bbeta^{\intercal}_{d_2}
\end{pmatrix} \right)
\ \text{ for any } \
\bw = (w_1, \ldots, w_m) \in \mathbb{R}^m.
\end{equation}

\smallskip

(b)
We now identify the set of vectors $\bw \in \mathbb{R}^m$ such that $\bJ'_{\bw} \in \tJ (G')$.
From Definition \ref{def:tilde_j_G} and assumption, $\tJ (G')$ denotes the set of vectors $\bJ' = (\bJ'_{\by_i \to \by_j})_{\by_i \to \by_j \in E'}$ satisfying that for any $1 \leq i \leq \ell$,
\begin{equation} \label{eq10:codimen_sum}
\sum_{\by \to \by_0 \in E_i} J'_{\by \to \by_0} 
= \sum_{\by_0 \to \by' \in E_i} J'_{\by_0 \to \by'}
\ \text{ for every vertex $\by_0 \in V_i$}.
\end{equation}
For each $1 \leq i \leq \ell$, let $\bJ'_{\bw (i)} \in \mathbb{R}^{r_i}$ denote the components of $\bJ'_{\bw}$ corresponding to the reactions within the $i$-th linkage class $L_i$, and define the matrix $\tilde{A}^{(i)} \in \mathbb{R}^{m_i \times r_i}$ as follows:
\[
\tilde{A}^{(i)}_{jk} =
\begin{cases}
1 & \text{ if $\by_{i,j}$ is the source vertex in the $k$-th reaction}, \\
-1 & \text{ if $\by_{i,j}$ is the target vertex in the $k$-th reaction}, \\
0 & \text{ otherwise}.
\end{cases}
\]
From \eqref{eq10:codimen_sum}, $\bJ'_{\bw} \in \tJ (G')$ is equivalent to the following condition: 
\begin{equation} \label{eq11:codimen_sum}
\tilde{A}^{(i)} \bJ'_{\bw (i)} = \mathbf{0}
\ \text{ for every $1 \leq i \leq \ell$}.
\end{equation}
Recall the Kirchhoff matrix $\tilde{J} \in \mathbb{R}^{m \times m}$ defined in \eqref{def3:codimen_sum}. The condition \eqref{eq11:codimen_sum} can be rewritten as
\begin{equation} \notag
\tilde{J}
\begin{pmatrix}
w_1 \\
\vdots \\
w_{m}
\end{pmatrix} = \mathbf{0}.
\end{equation}
Therefore, $\bJ'_{\bw} \in \tJ (G')$ if and only if $\bw \in \ker (\tilde{J})$.
\end{proof}

From Lemma \ref{lem:jr_g'_g}, the dimension of $\mJ(G', G)$ is determined by the complex-balanced flux condition and the $\mathbb{R}$-realizable condition.
The following theorem establishes a key property: when computing $\dim (\mJ(G', G))$, these two conditions operate independently.

\begin{theorem}
\label{thm:main}

Let $G' = (V', E')$ be a weakly reversible E-graph and let $G = (V, E)$ be an E-graph. 
Assume that $\mJ (G', G) \neq \emptyset$, then
\begin{equation}
\label{eq:codimen_sum}
\codim \big( \tJ (G') \cap \tmJ (G', G) \big) = \codim \big( \tJ (G') \big) + \codim \big( \tmJ (G', G) \big).
\end{equation}
\end{theorem}

\begin{proof}

By direct computation,
\begin{equation} \label{eq1:codimen_sum}
\big( \tJ (G') \cap \tmJ (G', G) \big)^{\perp}
= \tJ (G')^{\perp} + \tmJ (G', G)^{\perp}.
\end{equation}
This implies that
\begin{equation} \notag
\codim \big( \tJ (G') \cap \tmJ (G', G) \big) 
\leq \dim \big( \tJ (G')^{\perp} \big) + \dim \big( \tmJ (G', G)^{\perp} \big) 
= d_1 + d_2.
\end{equation}

To prove \eqref{eq:codimen_sum}, it suffices to show that the above inequality holds as an equality.
We proceed by contradiction. Suppose that
\begin{equation} \notag
\codim \big( \tJ (G') \cap \tmJ (G', G) \big) 
< \dim \big( \tJ (G')^{\perp} \big) + \dim \big( \tmJ (G', G)^{\perp} \big).
\end{equation}
This, together with \eqref{eq1:codimen_sum}, shows that 
\begin{equation} \label{eq1.5:codimen_sum}
\{ \balpha_1, \ldots, \balpha_{d_1}, \bbeta_1, \ldots, \bbeta_{d_2} \}
\ \text{ are linearly dependent}.
\end{equation}

On the other hand, recall the matrices $\tilde{A}^{(i)} \in \mathbb{R}^{m_i \times r_i}$ for each $1 \leq i \leq \ell$, and the linear conditions \eqref{eq10:codimen_sum} on $\tJ(G')$ in Lemma \ref{lem:intermediate}.
Kirchhoff's junction rules indicate that every linkage class with $m_i$ vertices imposes $(m_i - 1)$ independent constraints among the conditions in \eqref{eq10:codimen_sum}, implying that
\[
\text{the row rank of $\tilde{A}^{(i)}$} = \rank (\tilde{A}^{(i)}) = m_i - 1
\ \text{ for every $1 \leq i \leq \ell$}.
\]
Since $\mathbf{1} = (1, \ldots, 1) \in \mathbb{R}^{m_i}$ belongs to the left null space of $\tilde{A}^{(i)}$ for each $1 \leq i \leq \ell$, it follows that for every $1 \leq i \leq \ell$,
\[
\text{the first $(m_i - 1)$ row vectors of $\tilde{A}^{(i)}$ are linearly independent}.
\]
Moreover, Kirchhoff's junction rules state that the constraints from distinct linkage classes are independent. This implies that the collection of these $(m_i - 1)$ row vectors of $\tilde{A}^{(i)}$ for each $1 \leq i \leq \ell$ forms a basis for $\tJ (G')^{\perp}$.
Without loss of generality, we select this collection of vectors as $\{ \balpha^{\intercal}_1, \ldots, \balpha^{\intercal}_{d_1} \}$, and thus
\begin{equation} \label{eq13.5:codimen_sum}
d_1 = \sum\limits^{\ell}_{i=1} (m_i - 1) = \sum\limits^{\ell}_{i=1} m_i - \ell.
\end{equation}
From \eqref{eq1.5:codimen_sum} and the fact that $\{ \balpha_i \}_{i=1}^{d_1}$ and $\{ \bbeta_j \}_{j=1}^{d_2}$ form bases for $\tJ (G')^{\perp}$ and $\tmJ (G', G)^{\perp}$, respectively, we assume without loss of generality that
\begin{equation} \label{eq14:codimen_sum}
\balpha_{d_1}
\text{ is linearly dependent with  $\{ \balpha_1, \ldots, \balpha_{d_1 - 1}, \bbeta_1, \ldots, \bbeta_{d_2} \}$}. 
\end{equation}
Using \eqref{eq9:codimen_sum} in Lemma \ref{lem:intermediate} and \eqref{eq14:codimen_sum}, we conclude that $\bJ'_{\bw} \in \ker (A)$ if and only if
\begin{equation} \label{eq15:codimen_sum}
\bJ'_{\bw} \in \ker
\begin{pmatrix}
\balpha^{\intercal}_1 \\
\vdots \\
\balpha^{\intercal}_{d_1 - 1}
\end{pmatrix}.
\end{equation}
From \eqref{def3:codimen_sum} in Lemma \ref{lem:intermediate}, the condition in \eqref{eq15:codimen_sum} is equivalent to the following:
\begin{equation} \notag
\tilde{J}_{(d_1 - 1)}
\begin{pmatrix}
w_1 \\
\vdots \\
w_{m}
\end{pmatrix} = \mathbf{0},
\end{equation}
where $\tilde{J}_{(d_1 - 1)} \in \mathbb{R}^{(d_1 - 1) \times m}$ denotes the submatrix consisting of the first $(d_1 - 1)$ row vectors of $\tilde{J}$. This, together with \eqref{def2:codimen_sum} and \eqref{eq15:codimen_sum}, shows that
\begin{equation} \label{eq12:codimen_sum}
\bJ'_{\bw} \in \tJ (G')
\ \text{ if and only if } \
\bw \in \ker (\tilde{J}_{(d_1 - 1)}).
\end{equation}
From $\tilde{J}_{(d_1 - 1)} \in \mathbb{R}^{(d_1 - 1) \times m}$ and \eqref{eq13.5:codimen_sum}, it follows that
\begin{equation} \label{eq13:codimen_sum}
\dim \big( \ker (\tilde{J}_{(d_1 - 1)}) \big) \geq m - (d_1 - 1) = \ell + 1.
\end{equation}

Recall from Lemma~\ref{lem:intermediate} that $\bJ'_{\bw} \in \tJ (G')$ if and only if $\bw \in \ker (\tilde{J})$.
Since $G'$ is weakly reversible and consists of $\ell$ linkage classes with $m$ vertices, Kirchhoff's junction rules show that there exists $\bv_1, \ldots, \bv_{\ell} \in \mathbb{R}^m_{\geq 0}$, such that
\begin{equation} \notag
\ker (\tilde{J}) = \spn \{ \bv_1, \ldots, \bv_{\ell} \}
\ \text{ and } \
\dim \big( \ker (\tilde{J}) \big) = \ell. 
\end{equation}
However, this further implies that $\rank (\tilde{J}) = m - \ell$, and 
\begin{equation} \notag
\bJ'_{\bw} \in \tJ (G')
\ \text{ if and only if } \
\bw \in \spn \{ \bv_1, \ldots, \bv_{\ell} \}.
\end{equation}
This contradicts the conditions specified in \eqref{eq12:codimen_sum} and \eqref{eq13:codimen_sum}. Therefore, we conclude the theorem.
\end{proof}

\begin{theorem}
\label{thm:dim}

Let $G' = (V', E')$ be a weakly reversible E-graph with $\ell$ linkage classes and let $G = (V, E)$ be an E-graph. 
Assume that $\mJ (G', G) \neq \emptyset$, then
\begin{equation} \label{eq:dim}
\dim (\mJ (G', G)) = \dim \big( \tmJ (G', G) \big) - |V'| + \ell.
\end{equation}
\end{theorem}

\begin{proof}

From Lemma \ref{lem:jr_g'_g}, $\tJ (G')$ and $\tmJ (G', G)$ are both linear subspaces of $\RR^{|E'|}$. Thus,
\[
\tJ (G') \cap \tmJ (G', G)
\ \text{is a linear subspaces of $\RR^{|E'|}$}.
\]
Together with \eqref{eq:codimen_sum} in Theorem \ref{thm:main}, this implies
\begin{equation} \label{eq1:dim}
\begin{split}
\dim \big( \tJ (G') \cap \tmJ (G', G) \big) 
& = |E'| - \codim \big( \tmJ (G', G) \big) - \codim \big( \tJ (G') \big)
\\& = \dim \big( \tmJ (G', G) \big) - \codim \big( \tJ (G') \big).
\end{split}
\end{equation}
From \eqref{eq13.5:codimen_sum}, we deduce that
\begin{equation} \label{eq2:dim}
\codim \big( \tJ (G') \big) = |V'| - \ell.
\end{equation}
Using \eqref{eq2:dim} and \eqref{eq:jr_g'_g} in Lemma \ref{lem:jr_g'_g}, we obtain
\[
\dim (\mJ (G', G)) = \dim \big( \tJ (G') \cap \tmJ (G', G) \big) = \dim \big( \tmJ (G', G) \big) - ( |V'| - \ell ),
\]
and thus we conclude \eqref{eq:dim}.
\end{proof}

We conclude this section by outlining an algorithm to compute $\dim(\tmJ(G', G))$ as given in \eqref{eq:dim} of Theorem \ref{thm:dim}, followed by a proof of its correctness.

\newpage

\begin{breakablealgorithm}
\caption{(Compute the dimension of $\tmJ(G', G)$).}
\label{algo:dim}

\begin{algorithmic}[1]

\State Let $G = (V, E)$ and $G' = (V', E')$ be two E-graphs. 
Assume that $\mJ (G', G) \neq \emptyset$.

\smallskip

\State 
For each vertex $\by \in V'$, collect the reaction vectors in $G$ with source vertex $\by$ and compute their span, denoted by $S_{\by}$.

\smallskip

\State 
Determine a basis for $S_{\by}^{\perp}$, denoted by $\{ \balpha_1, \ldots, \balpha_{p} \} \subset \mathbb{R}^n$.
Construct the matrix:
\begin{equation} \notag
A_{\by} = 
\begin{pmatrix}
\balpha^{\intercal}_1 \\[5pt]
\vdots \\[5pt]
\balpha^{\intercal}_{p}
\end{pmatrix} \in \mathbb{R}^{p \times n}.
\end{equation}

\smallskip

\State 
Collect the reaction vectors with source vertex $\by$ in $G'$, denoted by $\{ \bbeta_1, \ldots, \bbeta_{q} \}$. 
Construct the matrix:
\begin{equation} \notag
B_{\by} = 
\begin{pmatrix}
\bbeta_1 \ \ldots \ \bbeta_{q}
\end{pmatrix} \in \mathbb{R}^{n \times q}.
\end{equation}

\smallskip

\State 
Compute $\dim ( \ker (A_{\by} B_{\by}) )$. Iterate this process for all vertices in $V'$, then
\[
\dim(\tmJ(G', G)) = \sum\limits_{\by \in V'} \dim (\ker (A_{\by} B_{\by} )).
\]
\end{algorithmic}
\end{breakablealgorithm}

\bigskip

Now we prove the correctness of Algorithm~\ref{algo:dim} via the following lemma.

\begin{lemma}
\label{lem:dim_tjr_g'_g}

Let $G = (V, E)$ and $G' = (V', E')$ be two E-graphs. 
Assume $\mJ (G', G) \neq \emptyset$, and follow the notation in Algorithm \ref{algo:dim}. Then
\begin{equation} \label{eq:dim_tjr_g'_g}
\dim(\tmJ(G', G)) = \sum\limits_{\by \in V'} \dim (\ker (A_{\by} B_{\by} )).
\end{equation}
\end{lemma}

\begin{proof}

From Definitions \ref{def:flux_realizable} and \ref{def:tilde_j_G}, any flux system $(G', \bJ')$ with $\bJ' \in \tmJ(G', G)$ is $\RR$-realizable on $G$. Since $\mJ(G', G) \neq \emptyset$, there exists some $\bJ \in \mathbb{R}^{|E|}$ such that for every vertex $\by_0 \in V \cup V'$,
\begin{equation} \label{def1:tjr_g'_g}
\sum_{\by_0 \to \by' \in E'} J'_{\by_0 \to \by'} (\by' - \by_0)
= \sum_{\by_0 \to \by \in E} J_{\by_0 \to \by} (\by - \by_0).
\end{equation}
Moreover, every vertex in $V \setminus V'$ yields an empty sum from the left-hand side of \eqref{def1:tjr_g'_g}, so the vector $\bJ \in \mathbb{R}^{|E|}$ must yield a zero sum from the right-hand side of \eqref{def1:tjr_g'_g}.

Therefore, it suffices to identify $\bJ' \in \mathbb{R}^{|E'|}$ that satisfies \eqref{def1:tjr_g'_g} for the vertices in $V'$.
Moreover, for any vertex $\by_0 \in V'$, \eqref{def1:tjr_g'_g} is equivalent to 
\begin{equation} \label{eq1:dim_tjr_g'_g}
\sum_{\by_0 \to \by' \in E'} J'_{\by_0 \to \by'} (\by' - \by_0) \in S_{\by_0}.
\end{equation}
Let $\bJ'_{\by} = (\bJ'_{\by \to \by'})_{\by \to \by' \in E'}$ denote the components of $\bJ'$ corresponding to the reactions with the source vertex $\by$. Thus, condition \eqref{eq1:dim_tjr_g'_g} can be rewritten as
\begin{equation} \label{eq2:dim_tjr_g'_g}
\sum_{\by_0 \to \by' \in E'} J'_{\by_0 \to \by'} (\by' - \by_0) 
= B_{\by_0} \bJ'_{\by_0} \in \ker (A_{\by_0}).
\end{equation}
This implies that $ \bJ'_{\by_0} \in \ker (A_{\by_0} B_{\by_0})$.
Since $\bJ'$ is a collection of all components corresponding to the source vertices in $V'$, and these components partition $\bJ'$, it follows that \eqref{eq:dim_tjr_g'_g} holds.
\end{proof}

\begin{remark}

Given two E-graphs $G$ and $G'$, the set $\tmJ(G', G)$ collects the flux vectors $\bJ'$ of $G'$ such that $(G', \bJ')$ is $\RR$-realizable on $G$. Thus, every vertex $\by_0 \in V'$ satisfies that
\begin{equation} \notag
\sum_{\by_0 \to \by \in E} J_{\by_0 \to \by} 
(\by - \by_0) 
= \sum_{\by_0 \to \by' \in E'} J'_{\by_0 \to \by'} 
(\by' - \by_0).
\end{equation}
This implies that the \emph{net flux vector} from every vertex in $V'$ needs to be $\RR$-realizable on $G$.
The result of Lemma \ref{lem:dim_tjr_g'_g} can be interpreted geometrically. Due to dynamical equivalence and the differences in reactions between two E-graphs, each vertex in $V'$ possesses a certain degree of freedom regarding its realizability in $G$. The dimension of $\tmJ(G', G)$ represents the aggregate degrees of freedom across all vertices in $V'$.

\smallskip

For example, consider the two E-graphs $G, G'$ in Figure \ref{fig:bruseelator} of Example \ref{ex:Brusselator}. It can be verified that $\mJ (G', G) \neq \emptyset$.
\begin{enumerate}
\item[(a)] For vertices $X, 3Y, X+2Y$, the reactions originating from them are identical in both $G$ and $G'$. Thus, the degrees of freedom in selecting vectors correspond to the total number of edges originating from these vertices (i.e., $1+1+2=4$).

\item[(b)] For vertex $Y$, there are two more reactions originating from it in $G'$ than in $G$, namely $Y \to 3Y$ and $Y \to X$.
To ensure that the net flux vector from $Y$ is $\RR$-realizable on $G$, the reaction rate constants $k_{Y \to 3Y}$ and $k_{Y \to X}$ must satisfy a proportionality constraint. Consequently, the degrees of freedom in selecting vectors correspond to the number of edges from $Y$ in $G'$ decreased by one due to the restriction (i.e., $3-1=2$).
\end{enumerate}
Finally, the dimension of $\tmJ(G', G)$  is the sum of the degrees of freedom from all vertices, which gives $4+2=6$ (see more details in Example \ref{ex:Brusselator}).
\end{remark}

\section{Applications}
\label{sec:applications}

In this section, the results on the disguised toric locus are applied to examples motivated by biochemical and cellular processes.
Specifically, the dimension of the disguised toric locus is explicitly computed for these examples, revealing that it attains full dimension within the corresponding parameter space.

\begin{example}[Brusselator-type models {\cite{brusselator,prigogine1968symmetry,boros2022limit}}]
\label{ex:Brusselator}

The Brusselator-type models are used to describe autocatalytic reactions. 
A specific modification of this model, known as the Selkov model, is found in the glycolysis pathway and is governed by the following reactions:
\begin{equation} \notag
X \rightleftharpoons X + 2Y, \ \
3Y \rightarrow Y \rightarrow X + 2Y \rightarrow 3Y, 
\end{equation}
where $X$ denotes Glucose, and $Y$ denotes ADP. This network is illustrated in Figure~\ref{fig:bruseelator}(a) and is denoted by $G$.
We further consider a weakly reversible E-graph $G' \sqsubseteq G_{c}$ in Figure~\ref{fig:bruseelator}(b) and compute $\dim(\dK(G, G'))$.

\begin{figure}[!ht]
\centering
\includegraphics[scale=0.4]{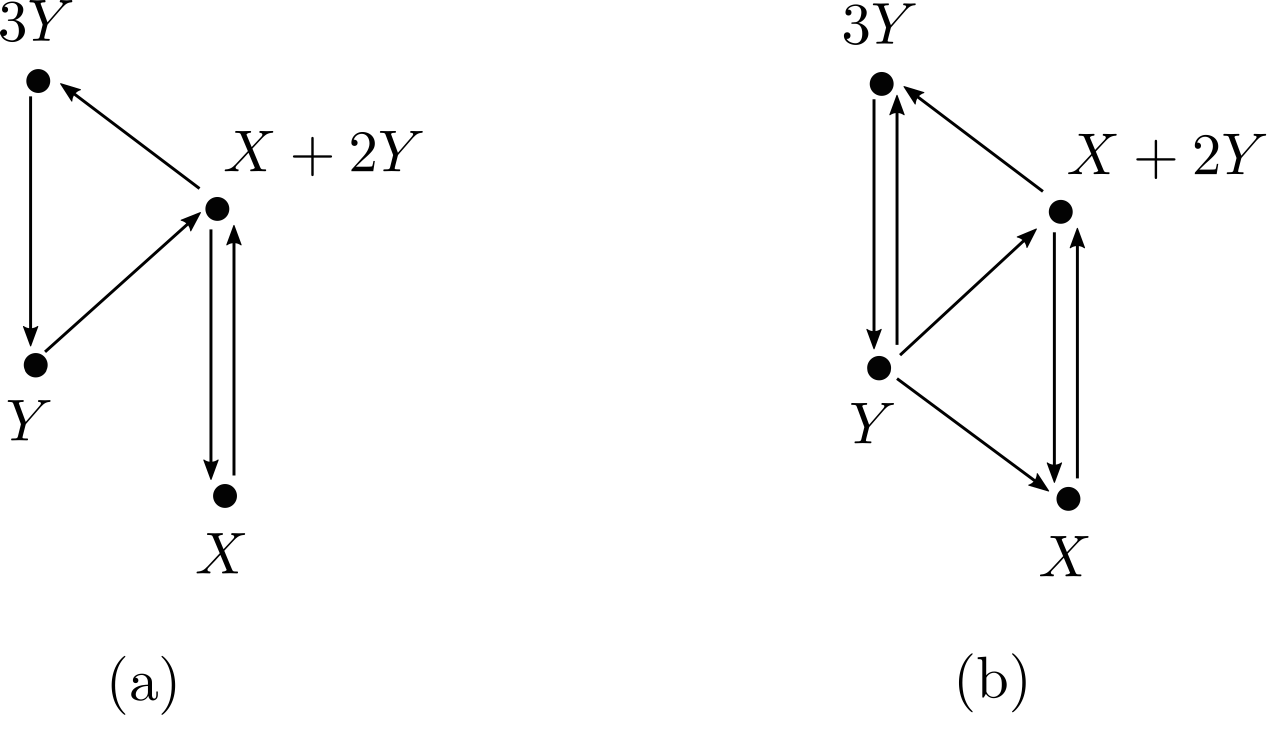}
\caption{(a) A version of the Selkov model, denoted by $G = (V, E)$. 
(b) The network $G' = (V', E')$, a weakly reversible subgraph of the complete graph associated with the source vertices of $G$.}
\label{fig:bruseelator}
\end{figure} 

We start with computing $\dim(\tmJ (G', G))$. Lemma~\ref{lem:dim_tjr_g'_g} shows that
\begin{equation} \notag
\dim(\tmJ(G', G)) = \sum\limits_{\by \in V'} \dim (\ker (A_{\by} B_{\by} )).
\end{equation}
Based on two E-graphs $G$ and $G'$ in Figure~\ref{fig:bruseelator}, we compute
\begin{equation} \notag
\begin{split}
A_X B_X = 
\begin{pmatrix}
1 & 0 
\end{pmatrix} 
\begin{pmatrix}
0 \\
2
\end{pmatrix}
= 0, &
\ \
\dim(\ker(A_X B_X)) = 1,
\\ A_Y B_Y = 
\begin{pmatrix}
-1 & 1 
\end{pmatrix} 
\begin{pmatrix}
1 & 1 & 0\\
-1 & 1 & 2
\end{pmatrix}
= 
\begin{pmatrix}
-2 & 0 & 2
\end{pmatrix}, &
\ \
\dim(\ker(A_Y B_Y)) = 2,
\\ A_{X+2Y} B_{X+2Y} = 
\begin{pmatrix}
0 & 0 
\end{pmatrix} 
\begin{pmatrix}
0 & -1 \\
-2 & 1
\end{pmatrix}
= 
\begin{pmatrix}
0 & 0
\end{pmatrix}, &
\ \
\dim(\ker(A_{X+2Y} B_{X+2Y})) = 2,
\\ A_{3Y} B_{3Y} = 
\begin{pmatrix}
1 & 0 
\end{pmatrix} 
\begin{pmatrix}
0 \\
-1
\end{pmatrix}
= 0, &
\ \
\dim(\ker(A_{3Y} B_{3Y})) = 1.
\end{split}
\end{equation}
This indicates that
\begin{equation} \notag
\begin{split}
\dim(\tmJ (G', G)) = 1 + 2 + 2 + 1 = 6.
\end{split}
\end{equation}
We then apply Theorem~\ref{thm:dim} to $G'$ and obtain
\[
\dim (\mJ(G', G)) 
= \dim \big( \tmJ (G', G) \big) - |V'| + \ell
= 6 - 4 + 1 = 3.
\]
Using Definition \ref{def:mas_ds}, Remarks \ref{rmk:d0} and \ref{rmk:j0}, we compute
\[
\dim (\mathcal{S}_{G'}) = 2, \ \
\dim (\mD (G)) = 0, \ \
\dim (\eJ(G')) =  0.
\]
It can be checked that $\pK (G, G') \neq \emptyset$. From Theorem \ref{thm:dim_kisg}, we derive that
\begin{equation}
\begin{split} \notag
\dim(\pK (G, G')) 
& = \dim(\dK(G, G')) 
\\& = \dim (\mJ(G', G)) + \dim (\mathcal{S}_{G'}) + \dim(\mD (G)) - \dim(\eJ (G')) = 5.
\end{split}
\end{equation}
Since $\pK (G) \subseteq \mathbb{R}^5_{>0}$, Theorem \ref{thm:dim_kisg_main} implies that $\dim (\pK (G)) = 5$.
\qed
\end{example}

\begin{example}[Thomas-type models {\cite[Chapter 6]{murray2002introduction}}]
\label{ex:thomas}

The Thomas-type models describe the catalytic reaction between uric acid and oxygen, with the enzyme uricase serving as a catalyst in this process. The model is governed by the following reactions:
\begin{equation} \notag
Y \rightleftharpoons \emptyset \rightleftharpoons X, \ \
Y \leftarrow X + Y \rightarrow X, 
\end{equation}
where $X$ denotes uric acid, and $Y$ denotes oxygen. This network is illustrated in Figure~\ref{fig:thomas_model}(a) and is denoted by $G$.
We further consider a weakly reversible E-graph $G' \sqsubseteq G_{c}$ in Figure~\ref{fig:thomas_model}(b) and compute $\dim(\dK(G, G'))$.

\begin{figure}[!h]
\centering
\includegraphics[scale=0.6]{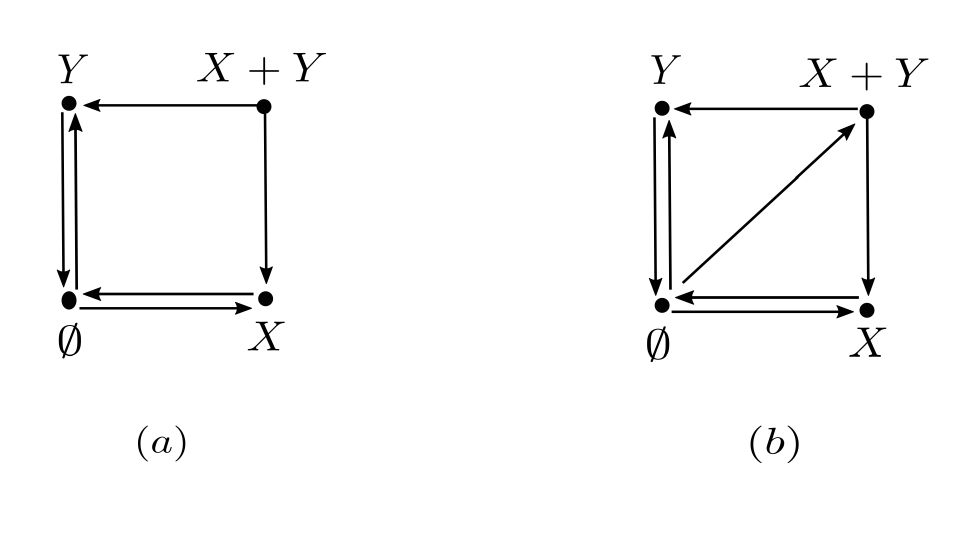}
\caption{(a) The Thomas-type model, denoted by $G = (V, E)$. 
(b) The network $G' = (V', E')$, a weakly reversible subgraph of the complete graph associated with the source vertices of $G$.}
\label{fig:thomas_model}
\end{figure}

Lemma~\ref{lem:dim_tjr_g'_g} shows that
\begin{equation} \notag
\dim(\tmJ(G', G)) = \sum\limits_{\by \in V'} \dim (\ker (A_{\by} B_{\by} )).
\end{equation}
Based on two E-graphs $G$ and $G'$ in Figure~\ref{fig:thomas_model}, we compute
\begin{equation} \notag
\begin{split}
A_\emptyset B_\emptyset = 
\begin{pmatrix}
0 & 0
\end{pmatrix} 
\begin{pmatrix}
1 & 0 & 1\\
0 & 1 & 1
\end{pmatrix}
= \begin{pmatrix}
0 & 0 & 0
\end{pmatrix}, &
\ \
\dim(\ker(A_\emptyset B_\emptyset)) = 3,
\\ A_X B_X = 
\begin{pmatrix}
0 & 1 
\end{pmatrix} 
\begin{pmatrix}
-1 \\
0
\end{pmatrix}
= 0, &
\ \
\dim(\ker(A_X B_X)) = 1,
\\ A_Y B_Y = 
\begin{pmatrix}
1 & 0 
\end{pmatrix} 
\begin{pmatrix}
0\\
-1
\end{pmatrix}
= 
\begin{pmatrix}
0
\end{pmatrix}, &
\ \
\dim(\ker(A_Y B_Y)) = 1,
\\ A_{X+Y} B_{X+Y} = 
\begin{pmatrix}
0 & 0 
\end{pmatrix} 
\begin{pmatrix}
-1 & 0\\
0 & -1
\end{pmatrix}
= 
\begin{pmatrix}
0 & 0
\end{pmatrix}, &
\ \
\dim(\ker(A_{X+2Y} B_{X+2Y})) = 2.
\end{split}
\end{equation}
This indicates that
\begin{equation} \notag
\begin{split}
\dim(\tmJ (G', G)) = 3 + 1 + 1 + 2 = 7.
\end{split}
\end{equation}
We then apply Theorem~\ref{thm:dim} to $G'$ and obtain
\[
\dim (\mJ(G', G)) 
= \dim \big( \tmJ (G', G) \big) - |V'| + \ell
= 7 - 4 + 1 = 4.
\]
Using Definition \ref{def:mas_ds}, Remarks \ref{rmk:d0} and \ref{rmk:j0}, we compute
\[
\dim (\mathcal{S}_{G'}) = 2, \ \
\dim (\mD (G)) = 0, \ \
\dim (\eJ(G')) =  0.
\]
It can be checked that $\pK (G, G') \neq \emptyset$. From Theorem \ref{thm:dim_kisg}, we derive that
\begin{equation}
\begin{split} \notag
\dim(\pK (G, G')) 
& = \dim(\dK(G, G')) 
\\& = \dim (\mJ(G', G)) + \dim (\mathcal{S}_{G'}) + \dim(\mD (G)) - \dim(\eJ (G')) = 6.
\end{split}
\end{equation}
Since $\pK (G) \subseteq \mathbb{R}^6_{>0}$, Theorem \ref{thm:dim_kisg_main} implies that $\dim (\pK (G)) = 6$.
\qed
\end{example}

\begin{example}[Circadian clock models \cite{leloup1999chaos}]
\label{ex:circadian}

The circadian clock regulates the body’s rhythms over a 24-hour cycle. A specific model of the circadian clock, introduced in~\cite{leloup1999chaos}, is governed by the following reactions:
\begin{equation} \notag
P + T \rightleftharpoons C \rightarrow \emptyset, \ \
P \rightleftharpoons \emptyset \rightleftharpoons T,
\end{equation}
where $P$ denotes period, $T$ denotes time, and $C$ denotes the period-time complex.
This network is illustrated in Figure~\ref{fig:circadian_clock}(a) and is denoted by $G$.
We further consider a weakly reversible E-graph $G' \sqsubseteq G_{c}$ in Figure~\ref{fig:circadian_clock}(b) and compute $\dim(\dK(G, G'))$.

\begin{figure}[!ht]
\centering
\includegraphics[scale=0.45]{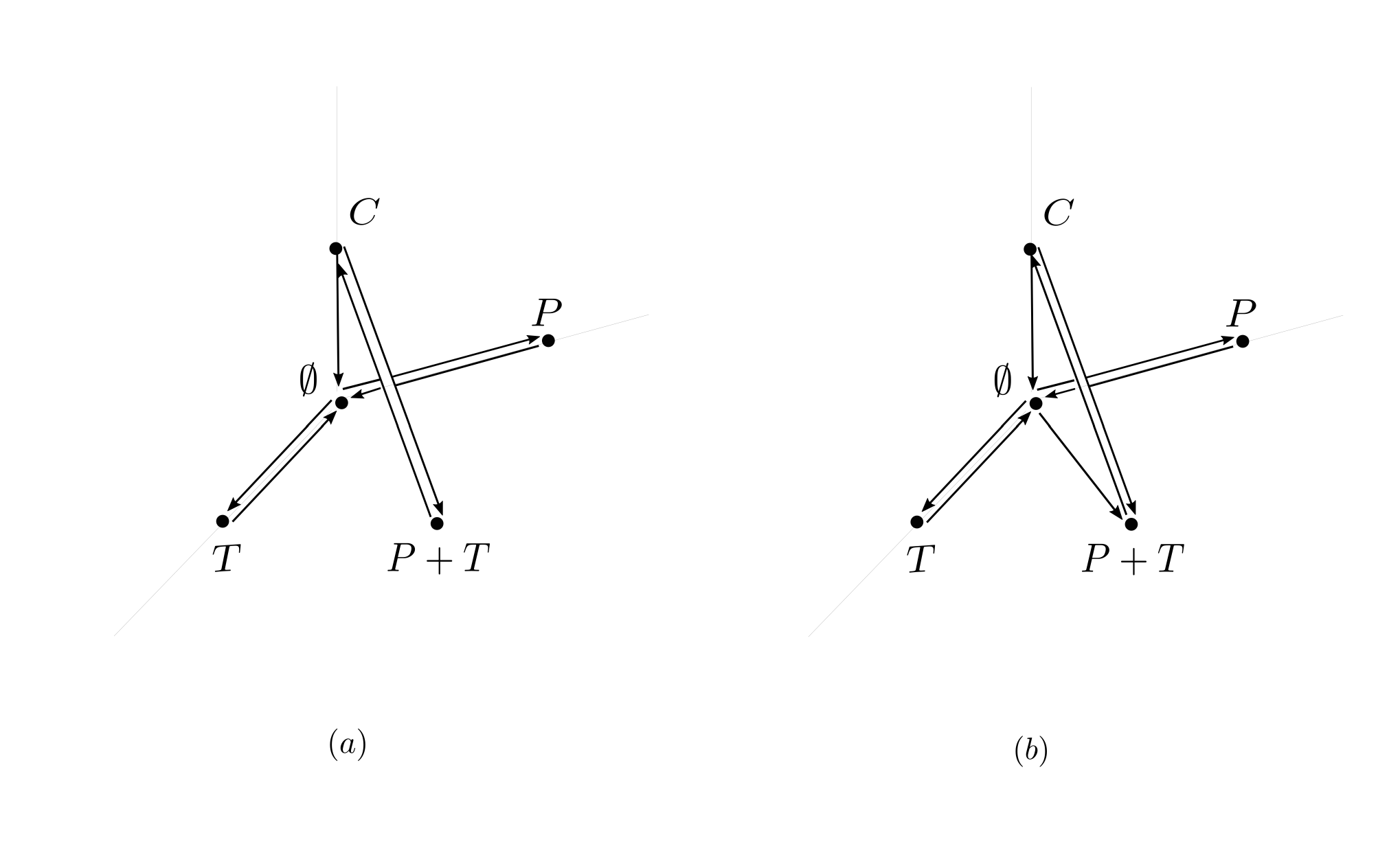}
\caption{(a) The circadian clock model, denoted by $G = (V, E)$. (b) The network $G' = (V', E')$, a weakly reversible subgraph of the complete graph associated with the source vertices of $G$.}
\label{fig:circadian_clock}
\end{figure}

Lemma~\ref{lem:dim_tjr_g'_g} shows that
\begin{equation} \notag
\dim(\tmJ(G', G)) = \sum\limits_{\by \in V'} \dim (\ker (A_{\by} B_{\by} )).
\end{equation}
Based on two E-graphs $G$ and $G'$ in Figure~\ref{fig:circadian_clock}, we compute
\begin{equation} \notag
\begin{split}
A_\emptyset B_\emptyset = 
\begin{pmatrix}
0 & 0 & 1
\end{pmatrix} 
\begin{pmatrix}
1 & 0 & 1\\
0 & 1 & 1\\
0 & 0 & 0
\end{pmatrix}
= \begin{pmatrix}
0 & 0 & 0
\end{pmatrix}, &
\ \
\dim(\ker(A_\emptyset B_\emptyset)) = 3,
\\ A_P B_P = 
\begin{pmatrix}
0 & 1 & 0\\
0 & 0 & 1
\end{pmatrix} 
\begin{pmatrix}
-1 \\
0 \\
0
\end{pmatrix}
= \begin{pmatrix}
0 \\
0
\end{pmatrix}, &
\ \
\dim(\ker(A_P B_P)) = 1,
\\ A_T B_T = 
\begin{pmatrix}
1 & 0 & 0\\
0 & 0 & 1
\end{pmatrix} 
\begin{pmatrix}
0\\
-1\\
0
\end{pmatrix}
= 
\begin{pmatrix}
0 \\
0
\end{pmatrix}, &
\ \
\dim(\ker(A_T B_T)) = 1,
\end{split}
\end{equation}
and
\begin{equation} \notag
\begin{split}
A_C B_C = 
\begin{pmatrix}
1 & -1 & 0
\end{pmatrix} 
\begin{pmatrix}
0 & 1 \\
0 & 1 \\
-1 & -1
\end{pmatrix}
= 
\begin{pmatrix}
0 & 0
\end{pmatrix}, &
\ \
\dim(\ker(A_C B_C)) = 2,
\\ A_{P+T} B_{P+T} = 
\begin{pmatrix}
-1 & 1 & 0\\
1 & 0 & 1
\end{pmatrix}
\begin{pmatrix}
-1 \\
-1 \\
1
\end{pmatrix} 
= 
\begin{pmatrix}
0 \\
0
\end{pmatrix}, &
\ \
\dim(\ker(A_{P+T} B_{P+T})) = 1.
\end{split}
\end{equation}
This indicates that
\begin{equation} \notag
\begin{split}
\dim(\tmJ (G', G)) = 3 + 1 + 1 + 2 + 1 = 8.
\end{split}
\end{equation}
We then apply Theorem~\ref{thm:dim} to $G'$ and obtain
\[
\dim (\mJ(G', G)) 
= \dim \big( \tmJ (G', G) \big) - |V'| + \ell
= 8 - 5 + 1 = 4.
\]
Using Definition \ref{def:mas_ds}, Remarks \ref{rmk:d0} and \ref{rmk:j0}, we compute
\[
\dim (\mathcal{S}_{G'}) = 3, \ \
\dim (\mD (G)) = 0, \ \
\dim (\eJ(G')) =  0.
\]
It can be checked that $\pK (G, G') \neq \emptyset$. From Theorem \ref{thm:dim_kisg}, we derive that
\begin{equation}
\begin{split} \notag
\dim(\pK (G, G')) 
& = \dim(\dK(G, G')) 
\\& = \dim (\mJ(G', G)) + \dim (\mathcal{S}_{G'}) + \dim(\mD (G)) - \dim(\eJ (G')) = 7.
\end{split}
\end{equation}
Since $\pK (G) \subseteq \mathbb{R}^7_{>0}$, Theorem \ref{thm:dim_kisg_main} implies that $\dim (\pK (G)) = 7$.
\qed
\end{example}

\section{Discussion}
\label{sec:discussion}

In this paper, we have developed methods to compute the dimension of the disguised toric locus of a reaction network. 
Our results build upon and advance prior work \cite{disg_3, disg_1, disg_2}. 
Notably, Theorem \ref{thm:main} shows the independence of the two key factors (complex-balanced conditions and $\RR$-realizable conditions) when determining $\dim(\mJ(G', G))$.
To demonstrate the applicability of our methods, we applied it to several models, including Brusselator-type models, Thomas-type models, and Circadian clock models, all of which have significant biological relevance.

While this work focuses primarily on computation, the approach introduced here opens several directions for future research. One promising direction is to identify conditions under which the dimension of the disguised toric locus of a reaction network attains full dimension in the parameter space. Recent work~\cite{disg_3} provided a formula for the dimension of the disguised toric locus. Building on this, we aim to establish the conditions on network structure that ensure the disguised toric locus forms an open subset in the parameter space.

Another potential research direction involves characterizing the conditions under which the disguised toric locus is a smooth manifold. In \cite{smooth2024}, it was established that the toric locus of a reaction network is a smooth manifold. 
Extending this result to the disguised toric locus can help in exploring the regularity of steady states within the disguised toric locus.

\section*{Acknowledgements}

This work was supported in part by the National Science Foundation grant DMS-2051568.

\bibliographystyle{unsrt}
\bibliography{Bibliography}

\end{document}